\documentclass[a4paper,11pt]{amsart}
\usepackage{latexsym,bm}
\usepackage{mathrsfs,amsmath,amssymb,amsthm,enumerate}
\usepackage[arrow,matrix]{xy}
\usepackage{color}
\usepackage{tikz}

\theoremstyle{plain}
\newtheorem{theorem}{Theorem}[section]
\newtheorem{proposition}[theorem]{Proposition}
\newtheorem{lemma}[theorem]{Lemma}
\newtheorem{corollary}[theorem]{Corollary}
\numberwithin{equation}{section}

\theoremstyle{definition}
\newtheorem{definition}[theorem]{Definition}
\newtheorem{remark}[theorem]{Remark}
\newtheorem{example}[theorem]{Example}
\newtheorem{question}[theorem]{Question}

\newcommand{\C}{\mathbb{C}}
\newcommand{\Q}{\mathbb{Q}}
\newcommand{\R}{\mathbb{R}}
\newcommand{\Z}{\mathbb{Z}}

\newcommand{\cF}{\mathcal{F}}
\newcommand{\cI}{\mathcal{I}}
\newcommand{\RZ}{\mathbb{R}\mathcal{Z}}
\newcommand{\B}{\mathcal{B}}

\newcommand{\Hom}{\operatorname{Hom}}

\def\red#1{{\textcolor{red}{#1}}}

\DeclareMathOperator{\rank}{rank}
\DeclareMathOperator{\row}{row}

\DeclareMathOperator{\wed}{wed}
\DeclareMathOperator{\link}{Lk}
\DeclareMathOperator{\Tor}{Tor}

\begin{document}
\title{On the cohomology and their torsion of real toric objects}

\author[S.Choi]{Suyoung Choi}
\address{Department of Mathematics, Ajou University, San 5, Woncheondong, Yeongtonggu, Suwon 443-749, Korea}
\email{schoi@ajou.ac.kr}

\author[H.Park]{Hanchul Park}
\address{Department of Mathematics, Ajou University, San 5, Woncheondong, Yeongtonggu, Suwon 443-749, Korea}
\email{hpark@ajou.ac.kr}

\thanks{This research was supported by Basic Science Research Program through the National Research Foundation of Korea(NRF) funded by the Ministry of Science, ICT \& Future Planning(NRF-2012R1A1A2044990) and the TJ Park Science Fellowship funded by the POSCO TJ Park Foundation.}

\date{\today}


\begin{abstract}

    In this paper, we do the two things.
    \begin{enumerate}
        \item We present a formula to compute the rational cohomology ring of a real topological toric manifold, and thus that of a small cover or a real toric manifold, which implies the formula of Suciu and Trevisan. Furthermore, the formula also works for other coefficient $\Z_q = \Z/q\Z$, where $q$ is a positive odd integer.
        \item We construct infinitely many real toric manifolds and small covers whose integral cohomology have a $q$-torsion for any positive odd integer $q$.
    \end{enumerate}
\end{abstract}
\maketitle

\tableofcontents
\section{Introduction}
    A \emph{toric variety}, which arose in the field of algebraic geometry, of dimension $n$ is a normal algebraic variety with an algebraic action of a complex torus $(\C^\ast)^n$ having a dense orbit. A compact smooth toric variety is sometimes called a \emph{toric manifold}. By regarding $S^1$ as the unit circle in $\C^\ast$, there is a natural action of $T^n = (S^1)^n\subset (\C^\ast)^n$ on a toric variety. Motivated by toric manifolds, some topological generalization of toric manifolds have been introduced and investigated well in toric topology. Instead of an algebraic torus action on an algebraic variety, one could think of a smooth torus ($T^n$ or $(\C^\ast)^n$) action on a smooth manifold.
    A \emph{quasitoric manifold} introduced in \cite{DJ91} is a closed smooth $2n$-manifold $M$ with an effective $T^n$-action such that
    \begin{enumerate}
      \item the torus action is \emph{locally standard}: i.e., it is locally isomorphic to the standard action of $T^n$ on $\R^{2n}$, and
      \item the orbit space $M/T^n$ can be identified with a simple polytope.
    \end{enumerate}
    A \emph{topological toric manifold} defined in \cite{IFM12} is a closed smooth $2n$-manifold $M$ with an effective smooth $(\C^\ast)^n$-action such that there is an open and dense orbit and $M$ is covered by finitely many invariant open subsets each of which is equivariantly diffeomorphic to a smooth representation space of $(\C^\ast)^n$. It was shown by \cite{IFM12} that every toric manifold or quasitoric manifold is a topological toric manifold.

    While the above toric objects are being considered, the notions of ``real toric objects'' are also studied. Let $M$ be a toric variety of complex dimension $n$. Then there is a canonical involution, called the \emph{conjugation} of $M$. The set of its fixed points, denoted by $M_{\R}$, is a real subvariety of dimension $n$, called a \emph{real toric variety}. When $M$ is a toric manifold, then $M_{\R}$ is a submanifold of dimension $n$ and called a \emph{real toric manifold}. This concept can be generalized as follows.

    \begin{definition}[\cite{DJ91}]
        A closed smooth manifold of dimension $n$ with a locally standard action of $\Z_2^n = (S^0)^n \subset (\R^\ast)^n$ is called a \emph{small cover} if its orbit space can be identified with a simple polytope.
    \end{definition}

    \begin{definition}[\cite{IFM12}]
        A closed smooth manifold of dimension $n$ with an effective smooth action of $(\R^\ast)^n$ is called a \emph{real topological toric manifold} if one of its orbits is open and dense and it is covered by finitely many invariant open subsets each of which is equivariantly diffeomorphic to a direct sum of real one-dimensional smooth representation spaces of $(\R^\ast)^n$.
    \end{definition}

    Similarly to toric objects, real toric objects have inclusions as the following diagram:
    
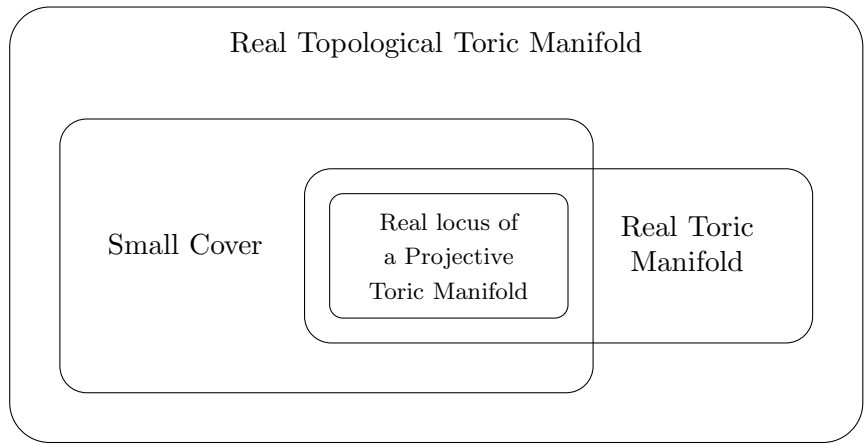
\begin{center}
    \begin{tikzpicture}[scale=0.165]
        \draw[rounded corners=15pt] (6,4) rectangle (74,39);
        \draw (40,36) node {Real Topological Toric Manifold};
        \draw[rounded corners=10pt] (10,8) rectangle (52.5,30);
        \draw (20,20) node [align=center]{Small Cover};
        \draw[rounded corners=10pt] (29.5,12) rectangle (70,26);
        \draw (60,20) node [align=center]{Real Toric \\Manifold};
        \draw[rounded corners=5pt] (31.5,14) rectangle (50.5,24);
        \draw (41,19) node [align=center]{\footnotesize Real locus of \\ \footnotesize a Projective \\ \footnotesize Toric Manifold};
    \end{tikzpicture}
    \end{center}

    It should be noted that the formulas for the integral cohomology ring of toric objects have been well established in many literatures such as \cite{Da78}, \cite{DJ91}, \cite{MP06}. Interestingly, the formula is quite simple; according to the formula, the ring is obtained as a quotient of a polynomial ring generated by only degree $2$ elements, and it has no torsion.

    Nevertheless, only little is known about the integral cohomology rings of real toric objects. Only the $\Z_2$-cohomology ring is known in \cite{DJ91}. \footnote{The formula in Theorem~4.14 of \cite{DJ91} is only for small covers. However, it is not hard to show that its argument is also applicable for real topological toric manifolds.} For example, it has been an open problem for a long time whether there is a real toric object whose cohomology admits an odd torsion or not.

    \begin{question}\label{que:cohomofsmallcover}
        For a real toric manifold (or a real topological toric manifold generally) $M$, find a formula to compute the cohomology ring $H^\ast(M;\Z)$.
    \end{question}

    \begin{question}\label{que:oddtorsion}
        Is there a real toric manifold $M$ such that $H^\ast(M;\Z)$ has an odd torsion?
    \end{question}

    Recently, Suciu and Trevisan \cite{Su1}, \cite{Trev} have announced that they established the formula for the rational cohomology of real toric objects using their description as covering spaces of real Davis-Januskiewicz spaces.

    In this paper, we present a formula for the cohomology ring of a real toric object with coefficient $\Q$ or $\Z_q$ ($q$: positive odd integer) by studying its natural cell structure. The explicit formula is stated in Theorem~\ref{thm:main1}. Furthermore, Theorem~\ref{thm:main1} implies the beautiful formula of Suciu and Trevisan (\cite{Su1}, \cite{Trev}), which is the $\Q$-coefficient case of Theorem~\ref{thm:cohomofsmallcover}. These theorems provide a partial answer to Question~\ref{que:cohomofsmallcover}.

    Moreover, using Theorem~\ref{thm:cohomofsmallcover}, we give a negative answer to Question~\ref{que:oddtorsion} by constructing infinitely many real toric manifolds whose cohomology have odd torsions. In fact, our examples are actually the real loci of  projective toric manifolds.

    The paper is organized as follows. In Section~\ref{sec:rttm}, we define real topological toric manifolds and real moment-angle complexes and roughly review their relation. In Section~\ref{sec:cohomofrzp}, we refer L.~Cai's work \cite{Ca} on the integral cohomology of the real moment-angle complex and introduce a classical result for cohomology of a finite regular cover so called the transfer homomorphism. The results about cohomology of real toric objects is covered in Section~\ref{sec:cohomofsmallcover}. Finally, we construct in Section~\ref{sec:oddtorsion} a real toric manifold whose cohomology contains an arbitrary odd torsion.

\section{Real topological toric manifolds}\label{sec:rttm}
    A \emph{simplicial complex} $K$ on a finite set $V=V(K)$ is a collection of subsets of $V$ satisfying
    \begin{enumerate}
      \item if $v \in V$, then $\{ v \} \in K$, and
      \item if $\sigma \in K$ and $\tau \subset \sigma$, then $\tau \in K$.
    \end{enumerate}

    A simplicial complex $K$ is called a \emph{simplicial sphere} of dimension $n-1$ if its geometric realization $|K|$ is homeomorphic to a sphere $S^{n-1}$, and is said to be \emph{star-shaped} if there are an embedding of $|K|$ into $\R^n$ and a point $p \in \R^n$ such that any ray from $p$ intersects $|K|$ once and only once.

    For a simplicial complex $K$ on $V=[m] = \{1,2,\dotsc,m\}$, we construct a topological space, called a \emph{real moment-angle complex} $\RZ_K$, as follows.
    \[
        \RZ_K = \bigcup_{\sigma\in K} \left\{(x_1,\dotsc,x_m)\in (D^1)^m \mid x_i \in S^0 \text{ when }i\notin \sigma\right\},
    \]
    where $D^1=[0,1]$ is the unit interval and $S^0 = \{0,1\}$ is its boundary. It should be noted that $\RZ_K$ is a manifold if $K$ is a simplicial sphere, and there is a canonical $\Z_2^m$-action on $\RZ_K$ which comes from the symmetry of $(D^1)^m$. See Chapter 6 of \cite{BP} for details.

    \begin{definition}[\cite{IFM12}]
        A closed smooth manifold $M$ of dimension $n$ with an effective smooth action of $(\R^\ast)^n$ is called a \emph{real topological toric manifold} if one of its orbits is open and dense and $M$ is covered by finitely many invariant open subsets each of which is equivariantly diffeomorphic to a direct sum of real one-dimensional smooth representation spaces of $(\R^\ast)^n$, where $\R^\ast = \R \setminus \{ 0\}$.
    \end{definition}

    Note that a real topological toric manifold $M$ has a natural $\Z_2^n$-action as a subgroup of $(\R^\ast)^n$. Then, its orbit space can be regarded as a manifold with corners, and it admits a natural face structure. It is known that its face structure is that of a star-shaped simplicial sphere $K$ of dimension $n-1$. We remark that a vertex of $K$ corresponds to a submanifold $M_i$ of $M$, called a \emph{characteristic submanifold}, which is fixed by a $\Z_2$-subgroup of $\Z_2^n$. We define a map, called a \emph{characteristic function of $M$}, $\lambda\colon V=[m] \to \Hom(\Z_2,\Z_2^n)\cong \Z_2^n$ such that $\lambda(i)$ fixes $M_i$ for all $i \in [m]$. The characteristic function satisfies the following \emph{non-singularity condition}:
    \begin{equation}\label{eqn:nonsingular}
     \text{$\lambda(i_1),\dotsc,\lambda(i_\ell)$ are linearly independent in $\Z_2^n$ if $\{i_1, \ldots, i_\ell\} \in K$. }\tag{$\ast$}
    \end{equation}

    Conversely, we can construct a real topological toric manifold (as a $\Z_2^n$-manifold) from the pair of $K$ and $\lambda$ above. For convenience, a characteristic function $\lambda$ is frequently represented by an $(n \times m)$-matrix $\Lambda = \left(\lambda(1)\;\dotsb\;\lambda(m)\right)$, called a \emph{characteristic matrix}. Define a map $\theta\colon [m] \to \Z_2^m$ so that $\theta(i)$ is the $i$-th coordinate vector of $\Z_2^m$. Then the map $\Lambda$ (as a matrix multiplication) satisfies $\Lambda \circ \theta = \lambda$ and the group $\ker \Lambda \cong \Z_2^{m-n}$ freely acts to $\RZ_K$. It is possible to show that its orbit space $M(K, \lambda)$ is a real topological toric manifold whose characteristic function is $\lambda$. Furthermore, it is shown by \cite{IFM12} that $M(K,\lambda)$ is Davis-Januszkiewicz equivalent to $M$. In other words, any real topological toric manifold of dimension $n$ as a $\Z_2^n$-manifold is determined by a pair $(K,\lambda)$ of a star-shaped simplicial sphere $K$ on $V$ of dimension $n-1$ and a map $\lambda \colon V \to \Z_2^n$ satisfying \eqref{eqn:nonsingular}.

    We remark that when $K$ is realizable as the dual complex of a simple polytope $P$, $M(K,\lambda)$ is known as a small cover (see \cite{DJ91}). In the case of the small cover, we do not distinguish the characteristic function $\lambda\colon V \to \Z_2^n$ from the characteristic function $\lambda\colon \cF \to \Z_2^n$ in the sense of \cite{DJ91} where $\cF$ is the set of facets of $P$. Note that the dual map is a bijection from $\cF$ to $V$.

%

\section{Cohomology of a real moment-angle complex}\label{sec:cohomofrzp}
    In this section, we briefly review a work of Cai \cite{Ca} describing the integral cohomology ring of $\RZ_K$. Regarding $D^1 = [0,1]$ as a CW complex consisting of two 0-cells $0,1$ and one 1-cell $\underline{01}$, the $m$-cube $(D^1)^m$ has a natural CW structure coming from the product operation. More precisely, let $D_i^1\cong [0,1]$ be the $i$-th factor of $(D^1)^m = D_1^1\times\dotsb\times D_m^1$ which is a CW complex with two 0-cells $0_i, 1_i$ and one 1-cells $\underline{01}_i$. Then every cell of $(D^1)^m$ is given as
    \[
        e_1\times\dotsb\times e_m, \quad e_i=0_i,\;1_i\;\text{or }\underline{01}_i.
    \]
    In this setting, observe that $\RZ_K$ is a subcomplex of $(D^1)^m$ and a cell $e = e_1\times\dotsb\times e_m$ of $(D^1)^m$ is a cell of $\RZ_K$ if and only if $\sigma_e := \{i\mid e_i = \underline{01}_i\}$ is a face of $K$ (see Lemma~1.2 of \cite{Ca}). Using this cell structure of $\RZ_K$, we are going to compute its cellular cohomology. When $X$ is a CW complex, let us denote by $C^*(X)$ be the cellular cochain complex of $X$. We equip $C^*(D_1^1)\otimes\dotsb\otimes C^*(D_m^1)$ with a differential $d$ such that
    \[
        d(e_1^*\otimes\dotsb\otimes e_m^*) = \sum_{i=1}^m(-1)^{\sum_{j=1}^{i-1}\deg e_j^*} e_1^*\otimes\dotsb\otimes e_{i-1}^*\otimes de_i^*\otimes e_{i+1}^*\otimes\dotsb\otimes e_m^*,
    \]
    where $e_i^*$ is the cochain dual to the cell $e_i$. Note that $d0_i^* = -\underline{01}_i^*$, $d1_i^* = \underline{01}_i^*$, and $d\underline{01}_i^* = 0$.
    Then we have a graded ring isomorphism
    \[
    \begin{array}{ccc}
        C^*((D^1)^m) & \to & C^*(D_1^1)\otimes\dotsb\otimes C^*(D_m^1) \\
        (e_1\times\dotsb\times e_m)^* &\mapsto & e_1^*\otimes\dotsb\otimes e_m^*
    \end{array}
    \]
    which preserves the differential. Thus from now on we will identify the cochain complex $C^*((D^1)^m)$ with $(C^*(D_1^1)\otimes\dotsb\otimes C^*(D_m^1),d)$.

    Now we perform a basis change of $C^*(D_i^1)$ by
    \begin{equation}\label{eqn:changeofbasis}
        \mathbf{1}_i = 0_i^*+1_i^*,\; t_i = 1_i^*,\; u_i = \underline{01}_i^*.
    \end{equation}
    From the definition of cup products, one can easily check the following relations:
    \begin{multline*}
        \mathbf{1}_i\cup t_i = t_i\cup\mathbf{1}_i,\quad \mathbf{1}_i\cup u_i = u_i\cup\mathbf{1}_i, \quad t_i\cup t_i = t_i, \\
        u_i\cup u_i = 0, \quad t_i\cup u_i = 0, \quad u_i\cup t_i = u_i.
    \end{multline*}

    We will use the notation $u_\sigma$ (respectively, $t_\sigma$) for the monomial $u_{i_1}\ldots u_{i_k}$ (respectively, $t_{i_1}\ldots t_{i_k}$) where $\sigma = \{i_1,\dotsc,i_k\}$, $i_1 <\dotsb < i_k$, is a subset of $[m]$. Let $\Z[u_1,\dotsc,u_m;t_1,\dotsc,t_m]$ be the differential graded ring with $2m$ generators such that
    \[
        \deg u_i = 1,\; \deg t_i = 0, \; du_i = 0, \; dt_i = u_i
    \]
    and the differential $d$ satisfies the Leibniz rule $d(ab) = da\cdot b + (-1)^{\deg a}a\cdot db$. Let $R$ be the quotient of $\Z[u_1,\dotsc,u_m;t_1,\dotsc,t_m]$ under the following relations
    \begin{multline*}
        u_it_i=u_i,\quad t_iu_i=0,\quad u_it_j = t_ju_i,\quad t_it_i = t_i,\\
         u_iu_i=0,\quad u_iu_j=-u_ju_i,\quad t_it_j=t_jt_i,
    \end{multline*}
    for $i,j=1,\dotsc,m$ and $i\ne j$. The Stanley-Reisner ideal $\mathcal{I}$ is the ideal generated by all square-free monomials $u_\sigma$ such that $\sigma$ is not a simplex of $K$. Note that, as an abelian group, $R/\mathcal{I}$ is freely spanned by the square-free monomials $u_\sigma t_{\omega\setminus \sigma}$, where $\sigma\subseteq \omega \subseteq [m]$ and $\sigma \in K$. 

    In his paper \cite{Ca}, Cai showed that there is a graded ring isomorphism preserving the differentials
    \[
        R/\mathcal{I} \cong C^\ast(\RZ_K)
    \]
    and thus
    \begin{theorem}[\cite{Ca}]
        There is a graded ring isomorphism
        \[
            H(R/\mathcal{I}, d) \cong H^*(\RZ_K).
        \]
    \end{theorem}

    We close this section with the following classical fact which is used later.
    \begin{theorem}[See III.2 of \cite{Br} or Section~3.G of \cite{Ha} for example]\label{thm:cohomoforbit}
        Let $X$ be a closed manifold and $\Gamma$ a finite group freely acting on $X$. Then there is a graded ring isomorphism
        \[
            H^\ast(X/\Gamma;G)\cong H^\ast(X;G)^\Gamma,
        \]
        when $G$ is a field of characteristic $0$ or coprime to $|\Gamma|$.
    \end{theorem}
    More precisely, let us assume that $X$ and $X/\Gamma$ are CW-complexes and the orbit map $\pi\colon X \to X/\Gamma$ preserves cells. Define a chain map of cellular chain complexes $\mu\colon C_\ast(X/\Gamma)\to C_\ast(X)$ such that $\mu(\pi(c)) = \sum_{g\in\Gamma} g\cdot c$ for any $c\in C_\ast(X)$ and hence one has the \emph{transfer homomorphism} $\mu^\ast\colon H^\ast(X;G) \to H^\ast(X/\Gamma;G)$ for any coefficient group $G$, such that
    \[
        \mu^\ast\pi^\ast = |\Gamma|\colon H^\ast(X/\Gamma;G) \to H^\ast(X/\Gamma;G)
    \]
    and
    \[
        \pi^\ast\mu^\ast = |\Gamma|\colon H^\ast(X;G)^\Gamma \to H^\ast(X;G)^\Gamma.
    \]
    From this fact, one observes
    \begin{corollary}\label{cor:cohomoforbitwhenGistwopower}
        Theorem~\ref{thm:cohomoforbit} also holds for any cyclic group $G=\Z_q$ of odd order if $|\Gamma|$ is a power of 2.
    \end{corollary}
    \begin{proof}
        It is enough to show that the maps $\mu^\ast\pi^\ast$ and $\pi^\ast\mu^\ast$ represented by the multiple $|\Gamma|$ are graded ring isomorphisms respectively. Recall the Fermat-Euler theorem which states that if $q$ and $a$ are coprime positive integers, then
        \[
            a^{\phi(q)}\equiv 1 \pmod q,
        \]
        where $\phi$ is Euler's phi function. Substituting $|\Gamma|$ for $a$ completes the proof since the map $|\Gamma|$ is invertible.
    \end{proof}
\section{Cohomology of a real topological toric manifold}\label{sec:cohomofsmallcover}

    Now let us consider the $\Z_2^m$-action on the cellular cochain complex $C^\ast(\RZ_K)$. Let $a_i$ be the $i$-th coordinate vector of $\Z_2^m$ for $1\le i\le m$. Then it is immediate from \eqref{eqn:changeofbasis} to observe that
    \begin{equation}\label{eqn:z2maction}
        a_i\cdot \mathbf{1}_j = \mathbf{1}_j, \quad a_i\cdot t_j = \left\{
                                                                     \begin{array}{ll}
                                                                       t_j, & \hbox{$i\ne j$;} \\
                                                                       \mathbf{1}_j-t_j, & \hbox{$i=j$,}
                                                                     \end{array}
                                                                   \right.
    \quad a_i\cdot u_j = \left\{
                                                                     \begin{array}{ll}
                                                                       u_j, & \hbox{$i\ne j$;} \\
                                                                       -u_j, & \hbox{$i=j$.}
                                                                     \end{array}
                                                                   \right.
    \end{equation}

    Recall that the group $\Gamma = \ker \Lambda\cong \Z_2^{m-n}$ freely acts to $\RZ_K$ and the action naturally induces  a $\Gamma$-action on $C^\ast(\RZ_K)\cong R/\mathcal{I}$. We define a bijection map $\varphi\colon \Z_2^m  \to 2^{[m]} $ such that the $i$-th entry of $v\in \Z_2^m$ is nonzero if and only if $i\in \varphi(v)$, where $2^{[m]}$ denotes the power set of $[m]$.

    \begin{lemma}\label{lem:rowlambda}
        Let $\row\Lambda $ be the row space of $\Lambda$ and $v$ a vector in $\Z_2^m$. Then $v\in\row\Lambda$ if and only if $| \varphi(v)\cap\varphi(u)|$ is even for all $u \in \ker \Lambda$.
    \end{lemma}
    \begin{proof} The lemma holds by the following observation;
        \begin{align*}
            v \in \row \Lambda
            &\Longleftrightarrow v\perp \ker \Lambda \\
            &\Longleftrightarrow v\cdot u = 0 \text{ for all } u\in \ker \Lambda \\
            &\Longleftrightarrow |\varphi(v)\cap \varphi(u)|\text{ is even for all } u\in \ker \Lambda.
        \end{align*}
    \end{proof}

    For $p \in R/\mathcal{I}$, we assign a $\Gamma$-invariant element of $R/\mathcal{I}$
    \[
        N(p) := \sum_{g\in \Gamma} g\cdot p.
    \]
    For a monomial $ x = u_\sigma t_{\omega\setminus \sigma}\in R/\mathcal{I}$ for $\sigma \subseteq \omega$, let us use the notation $b(x)=\omega$ for the union of all subscripts in $x$.  We give a partial order between monomials of $R/\cI$ such that $x \le y$ if and only if $b(x)\subseteq b(y)$. For a given polynomial $p \in R/\cI$, a term $x$ of $p$ is called \emph{maximal} in $p$ if it is maximal among the (nontrivial) terms of $p$.
    \begin{lemma}\label{lem:evenodd}
        Let $x$ be a monomial in $R/\cI$. In the polynomial $N(x)$, the coefficient of the term $x$ is nonzero if and only if $\varphi^{-1}(b(x)) \in \row \Lambda$. In that case, $|\Gamma|x= 2^{m-n}x$ is the unique maximal term in $N(x)$.
    \end{lemma}
    \begin{proof}
        The ``if'' part comes immediately by \eqref{eqn:z2maction} and Lemma~\ref{lem:rowlambda}. For the ``only if'' part, suppose that $\varphi^{-1}(b(x)) \notin \row \Lambda$. Then by Lemma~\ref{lem:rowlambda} there exists $u\in\ker\Lambda$ such that $|b(x)\cap \varphi(u)|$ is odd. Write $\ker \Lambda = \langle u \rangle \oplus V$ for some suitable subspace $V$. Reminding that
        $$\varphi(u+u_1) = \left( \varphi(u)\cup \varphi(u_1) \right) \setminus \left( \varphi(u)\cap \varphi(u_1) \right)$$
        for any $u_1 \in V$, observe that $|b(x)\cap \varphi(u_1)|$ is even if and only if $|b(x)\cap \varphi(u+u_1)|$ is odd and vice versa. This implies that the term $x$ is eliminated in the polynomial $N(x)$. The latter statement is obvious.
    \end{proof}

    \begin{figure}
            \begin{tikzpicture}[scale=1.3]
            \draw [fill=green] (0,0) rectangle (2,2);
            \draw (1,0-.4) node {$\binom01$};
            \draw (0-.3,1) node {$\binom10$};
            \draw (1,2+.4) node {$\binom01$};
            \draw (2+.3,1) node {$\binom11$};
            \draw (1,0+.25) node {$F_2$};
            \draw (0+.25,1) node {$F_1$};
            \draw (1,2-.25) node {$F_4$};
            \draw (2-.25,1) node {$F_3$};
        \end{tikzpicture}
        \caption{A characteristic function over the square.}\label{fig:k2char}
    \end{figure}
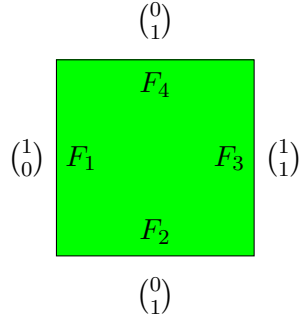
    \begin{example}
        Suppose that $P$ is a square as a simple polytope and a characteristic function $\lambda \colon \mathcal{F} \to \Z_2^2$ is given as Figure~\ref{fig:k2char}, where $\mathcal{F}$ is the facet set of $P$. One has
        \[
            \Lambda = \begin{pmatrix}
                1&0&1&0 \\
                0&1&1&1
            \end{pmatrix}
        \]
        and therefore
        $$\row\Lambda = \langle(1,0,1,0), (0,1,1,1)\rangle \text{ and } \ker\Lambda = \langle(1,1,1,0)^T , (0,1,0,1)^T\rangle.$$
        When $x=u_2t_{34}=u_2t_3t_4$, $\varphi^{-1}(b(x)) = (0,1,1,1)\in\row\Lambda$. Note that \begin{align*}
            N(x) &= u_2t_3t_4 - u_2(1-t_3)t_4 - u_2t_3(1-t_4) + u_2(1-t_3)(1-t_4) \\
            &= 4u_2t_3t_4 - 2u_2t_3 - 2u_2t_4 + u_2
        \end{align*}
        contains the term $4u_2t_3t_4 = |\Gamma|x$ as the unique maximal term.

        If $x=t_{123} = t_1t_2t_3$, then $\varphi^{-1}(b(x)) = (1,1,1,0)\notin \row\Lambda$. In this case, $N(x) = 2t_1t_3 -t_1 -t_3 + 1 $ has no term $t_{123}$.
    \end{example}
    Throughout the rest of this paper, we assume that $G$ is the coefficient ring (or group) $\Q$ or $\Z_q$ for  a positive odd integer $q$ unless otherwise mentioned. We denote by $(R/\cI\otimes G)|_\Lambda$ the graded $G$-subalgebra of $R/\cI\otimes G$ which is generated by $u_\sigma t_{\omega\setminus\sigma}$ for $\sigma\subseteq\omega\subseteq [m]$, $\sigma\in K$, and $\varphi^{-1}(b(x)) \in \row \Lambda$.
    \begin{proposition}\label{prop:RIfixedG}
        A polynomial $p\in R/\cI \otimes G$ is fixed by the $\Gamma$-action if and only if $p$ is a linear combination of polynomials of the form $N(x)$ where $x$ is a monomial in $R/\cI$ such that $\varphi^{-1}(b(x)) \in \row \Lambda$. In conclusion, there is a graded ring isomorphism
        \[
            (R/\cI\otimes G)^\Gamma \cong (R/\cI\otimes G)|_\Lambda.
        \]
    \end{proposition}
    \begin{proof}
        The ``if'' part is obvious. For the ``only if'' part, we first put $p_0=p$. Recursively take a maximal term $x_i$ in $p_i$ and put $p_{i+1} = p_i - N(x_i)/|\Gamma|$ which makes sense because $|\Gamma|$ is a unit in $G$ (see Corollary~\ref{cor:cohomoforbitwhenGistwopower}). Note that $\varphi^{-1}(b(x_i))\in \row \Lambda$ for all $i$: otherwise $N(p_i)/|\Gamma| = p_i$ could not contain the term $x_i$ by Lemma~\ref{lem:evenodd}, which is a contradiction. This process must end in finitely many steps, finishing the proof.
    \end{proof}
    Recall that $M(K,\lambda) = \RZ_K / \Gamma$ is the orbit space of the free $\Gamma$-action on $\RZ_K$.
    Hence, by applying Theorem~\ref{thm:cohomoforbit}, we obtain
    \begin{theorem}\label{thm:main1}
        There is a graded ring isomorphism
        \[
            H^\ast(M(K,\lambda);G) \cong H((R/\cI\otimes G)|_\Lambda, d).
        \]
        \qed
    \end{theorem}
    For $\omega \subseteq [m]$, let us denote by $R/\cI|_\omega$ the subgroup of $R/\cI$ generated by $u_\sigma t_{\omega\setminus\sigma}$ for $\sigma\subseteq\omega \subseteq [m]$ and $\sigma\in K$. The group $R/\cI|_\omega$ is closed under the differential $d$.
    Denote by $K_\omega= \{\sigma\in K \mid \sigma \subseteq \omega \}$ the full subcomplex of $K$ with respect to $\omega$.
    For each $\omega \subseteq [m]$, observe that there is a bijective cochain map of (co)chain complexes
    \begin{align*}
        f_\omega\colon R/\cI|_\omega \;&{\stackrel{\cong}{\longrightarrow}} \; C^\ast(K_\omega) \\
        u_\sigma t_{\omega\setminus \sigma} &\longrightarrow \sigma^\ast,
    \end{align*}
    where $C^\ast(K_\omega)$ means the simplicial cochain complex of $K_\omega$. This induces an additive isomorphism of cohomology
    \begin{equation}\label{eqn:pomegahomology}
        H^p(R/\cI|_\omega) \stackrel{\cong}{\longrightarrow} \widetilde{H}^{p-1} (K_\omega)
    \end{equation}
    and thus one concludes that there is an additive isomorphism
    \begin{equation}\label{eqn:rzkhochster}
         H^p(\RZ_K) \cong \bigoplus_{\omega\subseteq[m]} \widetilde{H}^{p-1}(K_\omega).
    \end{equation}
    See \cite[Theorem~1.6]{Ca} for details. We can give a similar argument for cohomology of $M(K,\lambda)$ with coefficient $G$. Since $(R/\cI\otimes G)|_\Lambda$ is the direct sum
    \[
        (R/\cI\otimes G)|_\Lambda = \bigoplus_{{\varphi^{-1}(\omega) \in \row \Lambda}\atop{\omega \subseteq [m] }} R/\cI|_\omega \otimes G
    \]
    as a graded abelian group, we deduce the following:
    \begin{theorem}\label{thm:cohomofsmallcover}
        There is an additive isomorphism
        \[
            H^p(M(K,\lambda);G) \cong \bigoplus_{{\varphi^{-1}(\omega) \in \row \Lambda}\atop{\omega \subseteq [m] }} \widetilde{H}^{p-1}(K_\omega;G).
        \]
        \qed
    \end{theorem}

\begin{remark}
    When $G=\Q$, the formula in the above theorem coincides with that due to Suciu and Trevisan (\cite{Su1}, \cite{Trev}).
\end{remark}

%

\section{Construction of real toric manifolds with odd torsions}\label{sec:oddtorsion}
    In this section, for a given positive odd number $q$, we construct a real topological toric manifold $M$ whose cohomology has a $q$-torsion. Furthermore, $M$ is realized as the real points of a projective toric manifold.

    First of all, we need the notion of \emph{nestohedra}. See \cite{P05}, \cite{PRW}, or \cite{Z06} for details.

    \begin{definition}
    A \emph{connected building set} $\mathcal{B}$ on a finite set $S$ is a collection of nonempty subsets of $S$ such that
        \begin{enumerate}
            \item $\mathcal{B}$ contains all singletons $\{i\}$, $i\in S$ and the entire set $S$,
            \item if $I, J \in \mathcal{B}$ and $I\cap J\neq\varnothing$, then $I\cup J\in\mathcal{B}$.
        \end{enumerate}
    \end{definition}
    For a connected building set $\B$, we can assign a simple polytope called a \emph{nestohedron}:
    \begin{definition}
        Let $\B$ be a connected building set on $[n+1]=\{1,\ldots,n+1\}$. For $I\subset [n+1]$, let $\Delta_I$ be the simplex given by the convex hull of $i$-th coordinate vectors for all $i\in I$. Then define \emph{the nestohedron} $P_\B$ as the Minkowski sum of simplices
        $$ P_\B = \sum_{I\in\B} \Delta_I $$
        where the Minkowski sum of two subsets $A,\,B\subset \R^n$ is defined by
        \[
            A+B := \{a+b \mid a\in A,\: b\in B\}.
        \]
    \end{definition}

    \begin{definition}
        For a connected building set $\B$ on $[n+1]$, a subset $N\subseteq \B\setminus\{[n+1]\}$ is called a \emph{nested set} if the following holds:
        \begin{enumerate}[{(N}1)]
            \item For any $I, J\in N$ one has either $I\subseteq J$, $J \subseteq I$, or $I$ and $J$ are disjoint.\label{n1}
            \item For any collection of $k\geq 2$ disjoint subsets $J_1, \ldots, J_k \in N$, their union $J_1\cup \cdots \cup J_k$ is not in $\B$.\label{n2}
        \end{enumerate}
        The nested set complex $\Delta_\B$ is defined to be the set of all nested sets for $\B$.
    \end{definition}

    We note that $\Delta_\B$ is surely a simplicial complex.

    \begin{theorem}\cite[Theorem 7.4]{P05}\label{thm:nested}
        Let $\B$ be a connected building set on $[n+1]$. Then $P_\B$ is a simple $n$-polytope and the dual complex of the nestohedron $P_\B$ is exactly $\Delta_\B$.
    \end{theorem}

    Furthermore, we have the following visual description of $P_\B$.

    \begin{proposition}\cite[Theorem 6.1]{Z06} \label{prop:obtain_P_B}
    Let $\B$ {be} a building set on $[n+1]$. Let $\varepsilon$ be a sequence of positive numbers $\varepsilon_1\ll\varepsilon_2\ll \ldots \ll \varepsilon_n\ll \varepsilon_{n+1}$. For each $I\in\B\setminus \{[n+1]\}$, assign a half-space
    $$
    A_I=\left\{(x_1,\ldots,x_{n+1})\in \R^{n+1}\middle| \sum_{i\in I}x_i \geq \varepsilon_{|I|}\right\}
    $$
    and for $I=[n+1]$, define $A_{[n+1]}$ be the hyperplane $x_1+\ldots +x_{n+1} = \varepsilon_{n+1}$. Let $P_\varepsilon$ be the intersection $\bigcap_{I\in\B} A_I$. Then one can choose $\varepsilon$ so that $P_\varepsilon=P_\B$ whose facets are given by $F_I=\partial A_I\cap P_\varepsilon$ for each $I\in\B\setminus\{[n+1]\}$. Furthermore, $P_\B$ is a Delzant polytope, i.e., the outward normal vector $\lambda(F)$ of each facet $F$ forms a basis at each vertex of $P_\B$.
    \end{proposition}

    In other words, $P_\B$ can be obtained by ``cutting off'' the faces of the $n$-simplex corresponding to the elements of $\B$ except the entire set $[n+1]$. Observe that $V(\Delta_\B) = \B\setminus \{[n+1]\}$ bijectively corresponds to the cutting hyperplanes.
    We sometimes use the notation like $124 = \{1,2,4\}$ for simplicity. When $\B = \{1,2,3,4,12,23,34,123,234,1234\}$, then $P_\B$ can be obtained as in Figure~\ref{fig:nestocut}.
    \begin{figure}[h]
        \begin{center}
    \begin{tikzpicture}[thick, scale=0.4]
    \draw (8,0)--(0,6)--(-3,-6)--cycle;
    \draw[dotted](0,0)--(8,0) (0,0)--(0,6) (0,0)--(-3,-6);
    \draw (1,1) node {\Huge{$2$}}
          (1.5,-2.2) node {\Large $1$}
          (3.5,2) node {\Large $3$}
          (-1,-.5) node {\Large $4$}
          (4,3.5) node {\Large{\red{$23$}}}
          (3,-3) node {\Large{\red{$12$}}};
    \draw[color=blue, ->] (-1.5,3)..controls (-1, 3.5)..(0,3);
    \draw (-2,3) node {\Large{\red{$34$}}}
          (0,6) node {\Large{\red{$234$}}}
          (8,0) node {\Large{\red{$123$}}};
    \end{tikzpicture}
    \begin{tikzpicture}[thick, scale=0.4]
    \draw (0,0) node {\Huge{\red{$2$}}}
          (4,3.5) node {\Large{\red{$23$}}}
          (0.5,-3.5) node {\Large{\red{$12$}}}
          (0.2,4.9) node {\Large{\red{$234$}}}
          (6,0) node {\Large{\red{$123$}}};
    \draw (5.3,1)--(5.3,-0.5)--(5.7,-2.3)--(7.2,0)--(6.3,2.5)--cycle;
    \draw (1.3,4.5)--(2.3,6)--(0.6,6.3)--(-0.7,5.5)--(-1.7,4)--cycle;
    \draw (6.3,2.5)--(2.3,6)
          (5.3,1)--(1.3,4.5)
          (5.3,-0.5)--(-4.3,-4.3)
          (5.7,-2.3)--(-3.5,-6)
          (-1.7,4)--(-4.3,-4.3)
          (-3.5,-6)--(-4.3,-4.3);
    \draw[dotted] (-3.5,-6)--(-0.7,-0.5)
          (0.8,0)--(-0.7,-0.5)
          (0.8,6.3)--(0.8,0)
          (-0.7,5.5)--(-0.7,-0.5)
          (0.8,0)--(7.2,0);
    \end{tikzpicture}
    \end{center}
    \caption{A 3-simplex and a nestohedron, before and after ``cutting''}\label{fig:nestocut}
    \end{figure}
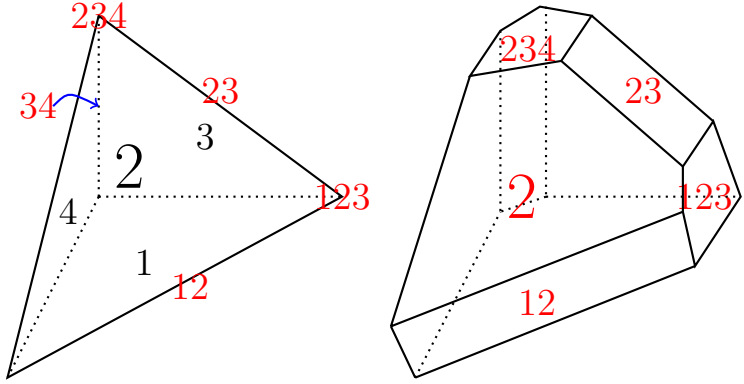

    For a set $A$ of subsets of $[n+1]$ such that the union of all elements of $A$ is the whole set $[n+1]$, one can define the connected building set \emph{generated by $A$} as the minimal connected building set on $[n+1]$ containing $A$ as a subset. Let $K$ be a given simplicial complex -- not necessarily a simplicial sphere -- on $[n+1]$. A subset $I$ of $[n+1]$ is called a \emph{non-face} of $K$ if it is not a face of $K$. It is \emph{minimal} if any proper subset of $I$ is a face of $K$. Let us denote by $\B(K)$ the connected building set generated by the minimal non-faces of $K$. The following lemma is immediately deduced by Theorem~\ref{thm:nested}.

    \begin{lemma}\label{lem:k_in_pb}
        Consider the subset $S = \{\{1\},\{2\}, \dotsc, \{n+1\}\}$ of all singletons of $V(\Delta_{\B(K)}) = {\B(K)} \setminus \{[n+1]\}$. Then the full subcomplex $\Delta_{\B(K)}|_S$ is isomorphic to $K$. \qed
    \end{lemma}

    \begin{remark}
        Lemma~\ref{lem:k_in_pb} is inspired by Theorem~11.11 and Theorem~11.12 of \cite{BM06}: roughly speaking, the polytope of \cite{BM06} is obtained by cutting off the faces of the $n$-simplex corresponding to the minimal non-faces of $K$ and one can further cut off some of its other faces to obtain our nestohedron $P_{\B(K)}$ (See Proposition~\ref{prop:obtain_P_B}). Our polytope $P_{\B(K)}$ has an advantage that it is uniquely determined independent of order of cutting and its face structure is well described. Furthermore, most importantly, $P_{\B(K)}$ supports a canonical real toric manifold explained later on.
    \end{remark}

    \begin{remark}
        Bosio and Meersseman used the construction of Theorem~11.11 and Theorem~11.12 of \cite{BM06} to show that the cohomology of an LV-M manifold may have arbitrary amount of torsion. One can easily observe that this construction also shows the analogous results for a moment-angle complex or a real moment-angle complex (refer \eqref{eqn:rzkhochster} for the real moment-angle complex case).
    \end{remark}

    \begin{definition}\cite[(4.1)]{CP12}
        Let us write $v_i$ for the $i$-th coordinate vector of $\Z_2^n$ for $1\le i \le n$ and $v_{n+1} = v_1 + \dotsb + v_n$.
        For given a connected building set $\B$ on $[n+1]$, define $\lambda \colon \B\setminus \{[n+1]\} \to \Z_2^n$ by
        \begin{equation}\label{eqn:lambdaofpb}
            \lambda(I) = \sum_{i \in I} v_i.
        \end{equation}
        Then $\lambda$ is a characteristic function for the polytope $P_\B$, defining a small cover of dimension $n$ denoted by $M_\R(\B)$. We call it the \emph{canonical real toric manifold} associated to $\B$. In fact, it is the real locus of the projective toric manifold defined by the Delzant polytope $P_\B$, hence it is a real toric manifold.
    \end{definition}

    For the next step, we introduce an operation of simplicial complexes called \emph{simplicial wedge operation}. Recall that for a face $\sigma$ of a simplicial complex $K$, the \emph{link} of $\sigma$ in $K$ is the subcomplex
    \[
         \link_K\sigma := \{ \tau \in K \mid \sigma\cup\tau\in K,\;\sigma\cap\tau=\varnothing\}
    \]
    and the \emph{join} of two disjoint simplicial complexes $K_1$ and $K_2$ is defined by
    \[
        K_1 \star K_2 = \{ \sigma_1 \cup \sigma_2 \mid \sigma_1 \in K_1,\; \sigma_2 \in K_2\}.
    \]
    Let $K$ be a simplicial complex with vertex set $V = [m]$ and fix a vertex $i$ in $K$. Consider a 1-simplex $I$ whose vertices are ${i_1}$ and ${i_2}$ and denote by $\partial I = \{i_1,\,i_2\}$ the 0-skeleton of $I$. Now, let us define a new simplicial complex on $m+1$ vertices, called the \emph{(simplicial) wedge} of $K$ at $i$, denoted by $\wed_{i}(K)$, by
    \[ \wed_{i}(K)= (I \star \link_K\{i\}) \cup (\partial I \star (K\setminus\{i\})), \]
    where $K\setminus\{i\}$ is the full subcomplex with $m-1$ vertices except $i$. The operation itself is called the \emph{simplicial wedge operation} or the \emph{(simplicial) wedging}. See Figure~\ref{fig:wedge}.

\begin{figure}[h]
    \begin{tikzpicture}[scale=.55]
        \coordinate [label=below:$1$](11) at (-9.8,0.6);
        \coordinate [label=right:$2$](22) at (-4,2.5);
        \coordinate [label=above:$3$](33) at (-6.5,4);
        \coordinate [label=above:$4$](44) at (-10.5,4);
        \coordinate [label=left:$5$](55) at (-12,2.5);
        \draw (-2,2.2) node {$\longrightarrow$};
        \coordinate [label={[xshift=-2.8pt,yshift=2.8pt]below:$1_1$}](1_1) at (2.2,0);
        \coordinate [label={[xshift=2.8pt,yshift=2.8pt]above:$1_2$}](1_2) at (2.2,2.2);
        \coordinate [label=right:$2$](2) at (8,3);
        \coordinate [label=above:$3$](3) at (5.5,4.5);
        \coordinate [label=above:$4$](4) at (1.5,4.5);
        \coordinate [label=left:$5$](5) at (0,3);
        \draw (-8,-1) node {$K$}
            (4,-1) node{$\wed_1(K)$};
        \draw [ultra thick] (11)--(22)--(33)--(44)--(55)--cycle
            (1_1)--(2)--(3)--(4)--(5)--cycle
            (1_2)--(1_1)
            (1_2)--(2)
            (1_2)--(3)
            (1_2)--(4)
            (1_2)--(5);
        \draw [thick, densely dashed] (4)--(1_1)--(3);
    \end{tikzpicture}
    \caption{Illustration of a wedge of $K$}\label{fig:wedge}
\end{figure}
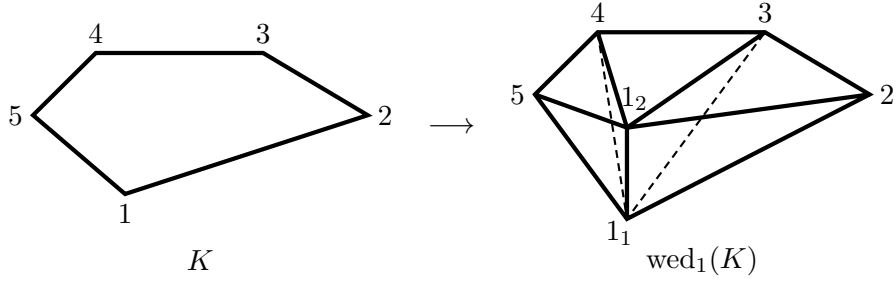

    We briefly present the notion of $K(J)$ of \cite{BBCG10} here. Let $K$ be a simplicial complex on vertices $V=[m]$. Note that a simplicial complex is determined by its minimal non-faces. Assign a positive integer $j_i$ to each vertex $i\in V$ and write $J=(j_1, \ldots, j_m)$. Denote by $K(J)$ the simplicial complex on vertices
    \[
        \{ \underbrace{1_1,1_2,\ldots,1_{j_1}},\underbrace{{2_1},2_2,\ldots,{2_{j_2}}},\ldots, \underbrace{{m_1},\ldots,{m_{j_m}}} \}
    \]
    with minimal non-faces
    \[
        \{ \underbrace{{(i_1)_1},\ldots,{(i_1)_{j_{i_1}}}},\underbrace{{(i_2)_1},\ldots,
        {(i_2)_{j_{i_2}}}},\ldots, \underbrace{{(i_k)_1},\ldots,{(i_k)_{j_{i_k}}}} \}
    \]
    for each minimal non-face $\{{i_1},\ldots,{i_k}\}$ of $K$. It is an easy observation to show that $\wed_{i}(K)=K(J)$ where $J$ is the $m$-tuple $(1,\ldots,1,2,1,\ldots,1)$ with 2 as the $i$-th entry. By consecutive application of wedgings on $K$, we can produce $K(J)$ for any $J$. Refer \cite{CP13} for more information about relations of wedge operations and toric objects. For the special case $J=(2,2,\dotsc,2)$, we write $K(J)=K(2,2,\dotsc,2)=: K'$. Note that $K'$ is obtained by wedging each vertex of $K$ once.

    \begin{theorem}\label{thm:smallcoverwithtorsion}
        Let $q$ be a positive odd number. Then there is a real toric manifold $M$, which is also a small cover, such that $H^\ast(M;\Z)$ has a $q$-torsion element.
    \end{theorem}

    \begin{proof}[proof of Theorem~\ref{thm:smallcoverwithtorsion}]
        It suffices to show when $q=p^k$ a power of an odd prime $p$. We take a simplicial complex $K$ with vertex set $[m]$ so that $H^\ast(K;\Z)$ has a $q$-torsion element. For example, one can use a triangulation of the mod $q$ Moore space $M(q,1)$. Then $K'$ has the cohomology with a $q$-torsion, since an wedge of $K$ is topologically the suspension
        \[
            |\wed_v K| \cong S|K|,\text{ for any $v\in [m]$,}
        \]
        and the suspension preserves torsion elements (possibly degree shifted) of cohomology. Now consider the connected building sets $\B(K)$ and $\B(K')$. Recall that $\B(K)$ is generated by the minimal non-faces of $K$. With the notation
        $$S = \{\{1\},\{2\}, \dotsc, \{m\}\}$$
        and
        $$S' = \{\{1_1\},\{1_2\},\{2_1\},\{2_2\},\dotsc,\{m_1\},\{m_2\}\},$$
        we see that $S \subset \B(K)$, $S'\subset \B(K')$, and there is a bijection
        \[
            \B(K)\setminus S \to \B(K')\setminus S'
        \]
        \[
            i_1i_2\dotsb i_\ell \mapsto (i_1)_1(i_1)_2(i_2)_1(i_2)_2\dotsb (i_\ell)_1(i_\ell)_2.
        \]
        Consider the characteristic matrix $\Lambda$ defined by \eqref{eqn:lambdaofpb} for the canonical real toric manifold $M = M_\R(\B(K'))$. The sum of all the $2m$ rows corresponds to the set $S'$ by the observation that every element of $\B(K')\setminus S'$ has even cardinality. By Lemma~\ref{lem:k_in_pb} and Theorem~\ref{thm:cohomofsmallcover}, the following inequality
        \[
            \rank_\Q H_\ast(M;\Q) < \rank_{\Z_{p^k}} H_\ast(M;\Z_{p^k})
        \]
        holds where the right hand side means the number of summands $\Z_{p^k}$ in $H_\ast(M;\Z_{p^k})$. A simple application of the universal coefficient theorem concludes the proof. More precisely, forgetting the grading, decompose the abelian group
        \[
            H_\ast(M;\Z) \cong \Z^b \oplus \Z_p^{c_1} \oplus \Z_{p^2}^{c_2} \oplus \dotsb \oplus \Z_{p^i}^{c_i} \oplus \dotsb \oplus \text{(other factors)}.
        \]
        Applying the universal coefficient theorem
        \[
            H_n(M;G) \cong H_n(M)\otimes G \oplus \Tor(H_{n-1}(M),G),
        \]
        we deduce that
        \[
            \rank_{\Z_{p^k}} H_\ast(M;\Z_{p^k}) = b+2c_k + 2c_{k+1}+2c_{k+2}+\dotsb.
        \]
        Since $\rank_\Q H_\ast(M;\Q) = b$, $c_i$ is nonzero for some $i\ge k$, which means that there is a $p^k$-torsion in $H_\ast(M;\Z)$.

    \end{proof}

    \begin{remark}
        There is an important subclass of the set of nestohedra, called graph associahedra (See \cite{CP12}, \cite{P05}, or \cite{PRW}). From results of \cite{CP12}, one can conclude that the canonical real toric manifold $H^\ast(M_\R(P);\Z)$ has no odd torsion when $P$ is a graph associahedron since for any $\omega\subseteq [m]$ such that $\varphi^{-1}(\omega) \in \row \Lambda$, the simplicial complex $K_\omega$ is a wedge of spheres. We remark that Theorem~\ref{thm:smallcoverwithtorsion} shows that an analogous result of this does not hold for general nestohedra instead of graph associahedra.
    \end{remark}


\end{document}